\documentclass[12pt]{article}%
\usepackage{amssymb}
\usepackage{amsfonts}
\usepackage{graphicx}
\usepackage{amsmath}%
\providecommand{\U}[1]{\protect\rule{.1in}{.1in}}
\newtheorem{theorem}{Theorem}

\newtheorem{corollary}[theorem]{Corollary}

\newtheorem{proposition}[theorem]{Proposition}
\newtheorem{remark}[theorem]{Remark}

\newenvironment{proof}[1][Proof]{\textbf{#1.} }{\ \rule{0.5em}{0.5em}}
\begin{document}

\title{On the critical length conjecture for spherical Bessel functions in CAGD}
\author{O. Kounchev and H. Render}
\maketitle

\begin{abstract}
A conjecture of J.M. Carnicer, E. Mainar and J.M. Pe\~{n}a states that the
critical length of the space $P_{n}\odot C_{1}$ generated by the functions
$x^{k}\sin x$ and $x^{k}\cos x$ for $k=0,...n$ is equal to the first positive
zero $j_{n+\frac{1}{2},1}$ of the Bessel function $J_{n+\frac{1}{2}}$ of the
first kind. It is known that the conjecture implies the following statement
(D3): the determinant of the Hankel matrix%
\begin{equation}
\left(
\begin{array}
[c]{ccc}%
f & f^{\prime} & f^{\prime\prime}\\
f^{\prime} & f^{\prime\prime} & f^{\left(  3\right)  }\\
f^{\prime\prime} & f^{\prime\prime\prime} & f^{\left(  4\right)  }%
\end{array}
\right) \label{eqabstract}%
\end{equation}
does not have a zero in the interval $(0,j_{n+\frac{1}{2},1})$ whenever
$f=f_{n}$ is given by $f_{n}\left(  x\right)  =\sqrt{\frac{\pi}{2}}%
x^{n+\frac{1}{2}}J_{n+\frac{1}{2}}\left(  x\right)  .$ In this paper we shall
prove (D3) and various generalizations.

\end{abstract}



{\large CORRESPONDING AUTHOR:} Ognyan Kounchev,

email:\ kounchev@abv.bg;

\textbf{Mailing address}: Institute of Mathematics and Informatics, Bulgarian
Academy of Sciences, Acad. G. Bonchev Str., bl. 8, 1113 Sofia, Bulgaria

{\large AUTHOR:} Hermann Render, email:\ hermann.render@ucd.ie;

\textbf{Mailing address}: School of Mathematics and Statistics, University
College Dublin, Belfield 4, Dublin, Ireland.

\bigskip

\bigskip

\section{Introduction}

The critical length $l\left(  U\right)  $ of a finite dimensional,
translation-invariant space $U$ of differentiable functions defined on the set
$\mathbb{R}$ of all real numbers is the supremum of the lengths of the
intervals where Hermite interpolation problems are unisolvent for any choice
of nodes (see \cite{CMP04}). Equivalently, the critical length is the supremum
of the lengths of the intervals where $U$ is an extended Chebyshev space.
Recall that a linear space $U$ of functions defined on an interval $I$ is an
extended Chebyshev space if the number of zeros (including multiplicities) in
$I$ of any nonzero function is less than the dimension of the space. The
notion of critical length is an important tool in CAGD and it is related to
the question which finite dimensional spaces containing oscillating functions
are suitable for design purposes. Oscillating functions are of interest since
they represent curves and solutions of differential equations related with the
physical nature of some problems, see e.g. \cite{AKR07}, \cite{AKR08},
\cite{AKR08b}, \cite{CMP94} , \cite{CMP04}, \cite{CMP23}, \cite{Pena97},
\cite{Zhan96}. A particular example is the cycloidal space $C_{n}$, the linear
space generated by the functions
\[
1,x,....,x^{n-2},\cos x,\sin x
\]
which has been investigated in \cite{CMP14}, \cite{CMP17}, \cite{AMR19}.
Another example, which will be the main subject in this paper, is the linear
space generated by
\[
\cos x,\sin x,\quad x\cos x,\quad x\sin x,\quad...\quad,x^{n}\cos x,\quad
x^{n}\sin x
\]
which is denoted by $P_{n}\odot C_{1}$ in \cite{CMP23}. Note that $C_{n}$ and
$P_{n}\odot C_{1}$ are subspaces of the space $C\left(  \mathbb{R}\right)  $
of all continuous real-valued functions defined on $\mathbb{R}$ which are
invariant under translations and reflections. This means that for any $u\in U
$ the translation $u_{h}$ defined by $u_{h}\left(  x\right)  =u\left(
x-h\right)  $ and the reflection $u_{r}$ defined by $u_{r}\left(  x\right)
=u\left(  -x\right)  $ are in $U$ for any real number $h.$

It is shown in the recent paper \cite{CMP17} (see also \cite{AMR19}) that the
critical length $l\left(  C_{n}\right)  $ of the cycloidal space $C_{n}$ is
related to the first positive zero $j_{\nu,1}$ of the Bessel function $J_{\nu
}$ of the first kind for $\nu=\nu_{n}:=\lfloor n/2\rfloor-\frac{1}{2},$ more
precisely the equation $l\left(  C_{n}\right)  =2j_{\nu_{n},1}$ holds, where
\[
J_{\nu}\left(  x\right)  :=%
%TCIMACRO{\dsum _{k=0}^{\infty}}%
%BeginExpansion
{\displaystyle\sum_{k=0}^{\infty}}
%EndExpansion
\frac{\left(  -1\right)  ^{k}}{\Gamma\left(  k+1\right)  \Gamma\left(
k+\nu+1\right)  }\left(  \frac{x}{2}\right)  ^{2k+\nu}\text{ }%
\]
for $\nu\in\mathbb{R}.$ The critical length of $P_{n}\odot C_{1}$ is discussed
in \cite{CMP23} and it is shown that it is bounded by the first positive zero
$j_{n+\frac{1}{2},1}$ of the Bessel function $J_{n+\frac{1}{2}}.$ J.M.
Carnicer, E. Mainar and J.M. Pe\~{n}a conjectured in \cite{CMP23} that this
inequality is actually an equality and we shall call this the CMP-conjecture.

Let us outline what it known about the CMP-conjecture: recall that a basis
$b_{0},...,b_{n}$ of a linear space $U\subset C\left[  a,b\right]  $ of
dimension $n+1$ is called canonical if $b_{k}$ has zero of exact order $k$ at
$x=a$ for $k=0,...,n$, so $b_{k}^{\left(  j\right)  }\left(  0\right)  =0$ for
$j=0,...,k-1,$ and $b_{k}^{\left(  k\right)  }\left(  0\right)  \neq0$ where
$f^{\left(  k\right)  }$ is the $k$-th derivative of a function $f.$
Proposition 3 in \cite{CMP23} states that the critical length of a linear
space invariant under translation and reflection and with a canonical basis
$b_{0},...,b_{n}$ is the least positive zero of the functions
\[
w_{j,n}\left(  x\right)  =\det W\left(  u_{j},...,u_{n}\right)  \left(
x\right)  \qquad j>n/2,
\]
where $W\left(  u_{j},...,u_{n}\right)  \left(  x\right)  $ is the Wronskian
matrix at $x$ of the system of functions $\left(  u_{j},...,u_{n}\right)  $
defined by
\[
W\left(  u_{j},...,u_{n}\right)  =\left(
\begin{array}
[c]{cccc}%
u_{j} & u_{j+1} & \cdots & u_{n}\\
u_{j}^{\prime} & u_{j+1}^{\prime} &  & u_{n}^{\prime}\\
\vdots &  &  & \vdots\\
u_{j}^{\left(  n-j\right)  } & u_{j+1}^{\left(  n-j\right)  } & \cdots &
u_{n}^{\left(  n-j\right)  }%
\end{array}
\right)  .
\]

Noting that the space $P_{n}\odot C_{1}$ has dimension $2n+2$ it is shown in
\cite[Proposition 3]{CMP23} that a canonical basis of $P_{n}\odot C_{1}$ is
given by the system of derivatives $f_{n}^{\left(  2n+1-k\right)  }$ for
$k=0,...,2n+1$ of the function
\begin{equation}
f_{n}\left(  x\right)  =\sqrt{\frac{\pi}{2}}x^{n+\frac{1}{2}}J_{n+\frac{1}{2}%
}\left(  x\right)  .\label{eqdeffn}%
\end{equation}
The function $f_{n}$ has a zero of order $2n+1$ at $x=0$ and it satisfies the
differential equation
\begin{equation}
f_{n}^{\prime\prime}\left(  x\right)  -\frac{2n}{x}f_{n}^{\prime}\left(
x\right)  +f_{n}\left(  x\right)  =0.\label{eqdefnewf}%
\end{equation}
When applying the above-mentioned Proposition 3 in \cite{CMP23} one has to
show that $w_{j,2n+1}\left(  x\right)  \neq0$ for $j>\left(  2n+1\right)  /2.$
Note that for $j=2n+1$ this means that
\[
w_{2n+1,2n+1}\left(  x\right)  =f_{n}\left(  x\right)
\]
has no zeros in the interval $\left(  0,b\right)  ,$ and the critical length
is smaller than $j_{n+\frac{1}{2},1}$. One key technical result shown in
\cite[Proposition 4]{CMP23} is the fact that for all natural numbers $n$ the
following determinants
\[
w_{2n,2n+1}\left(  x\right)  =\det\left(
\begin{array}
[c]{cc}%
f_{n}^{\prime}\left(  x\right)  & f_{n}\left(  x\right) \\
f_{n}^{\prime\prime}\left(  x\right)  & f_{n}^{\prime}\left(  x\right)
\end{array}
\right)  =f_{n}^{\prime}\left(  x\right)  f_{n}^{\prime}\left(  x\right)
-f_{n}\left(  x\right)  f_{n}^{\prime\prime}\left(  x\right)
\]
are positive on the entire interval $\left(  0,\infty\right)  $. Consequently
the CMP-conjecture is true for the case $n=1$, since the condition
$w_{j,3}\left(  x\right)  >0$ for all $j>3/2$ with $j\leq3$ (i.e. $j=2,3)$ is satisfied.

The main result of the present paper states that for all natural numbers $n$
the following determinant
\[
w_{2n-1,2n+1}=\det\left(
\begin{array}
[c]{ccc}%
f_{n}^{\prime\prime} & f_{n}^{\prime} & f_{n}\\
f_{n}^{\prime\prime\prime} & f_{n}^{\prime\prime} & f_{n}^{\prime}\\
f_{n}^{\left(  4\right)  } & f_{n}^{\prime\prime\prime} & f_{n}^{\prime\prime
}\,
\end{array}
\right)
\]
is positive on the interval $\left(  0,j_{n+\frac{1}{2},1}\right)  $.
Consequently the CMP -conjecture is true also for the case $n=2.$

Our methods of proof are applicable in a more general setting: instead of
working with the Bessel-type function $f_{n}$ we shall assume that a function
$f:\left(  a,b\right)  \rightarrow\mathbb{R}$ is given which satisfies a
differential equation
\begin{equation}
f^{\prime\prime}+pf^{\prime}+qf=0\label{eqODE}%
\end{equation}
where $p,q\in C^{4}\left(  a,b\right)  .$ We are investigating in Section 2
conditions on $p,q$ such that the function
\begin{equation}
v\left(  f\right)  :=\det\left(
\begin{array}
[c]{cc}%
f^{\prime} & f^{\prime\prime}\\
f & f^{\prime}%
\end{array}
\right)  =\left(  f^{\prime}\right)  ^{2}-f^{\prime\prime}f\label{eqdefv}%
\end{equation}
is positive and increasing on $\left(  a,b\right)  $, and we shall generalize
some results in \cite{CMP23} to this more general context.

In Section 3 we begin with the investigation of our main subject in this
paper, namely
\begin{equation}
w\left(  f\right)  :=\det\left(
\begin{array}
[c]{ccc}%
f^{\prime\prime} & f^{\prime} & f\\
f^{\left(  3\right)  } & f^{\prime\prime} & f^{\prime}\\
f^{\left(  4\right)  } & f^{\left(  3\right)  } & f^{\prime\prime}%
\end{array}
\right)  =-\det\left(
\begin{array}
[c]{ccc}%
f & f^{\prime} & f^{\prime\prime}\\
f^{\prime} & f^{\prime\prime} & f^{\left(  3\right)  }\\
f^{\prime\prime} & f^{\left(  3\right)  } & f^{\left(  4\right)  }%
\end{array}
\right) \label{eqdefw}%
\end{equation}
where in the last step we have just interchanged column 1 and 3. Using the
differential equation (\ref{eqODE}) we deduce that
\[
w\left(  f\right)  =\left(  p^{\prime}f^{\prime}+q^{\prime}f\right)
^{2}f-A\cdot v\left(  f\right)
\]
where $A$ is a suitable expression and $v\left(  f\right)  $ is given by
(\ref{eqdefv}), see Theorem \ref{ThmMain3}.

In order to discuss monotonicity properties of $w$ we investigate in Section 4
the derivative $w^{\prime}$. One key result is Theorem \ref{ThmMain4}: If
$q^{\prime}=0$ then the following identity
\[
w^{\prime}+\frac{3}{2}\left(  p-\frac{p^{\prime\prime}}{p^{\prime}}\right)
w=p^{\prime}\cdot f^{\prime}\cdot V
\]
holds where
\[
V:=\left(  \frac{1}{2}\left(  p^{\prime}p-p^{\prime\prime}\right)
f+p^{\prime}f^{\prime}\right)  f^{\prime}-\left(  \frac{1}{2}p^{2}-p^{\prime
}-2q+\frac{p^{\prime\prime\prime}}{p^{\prime}}-\frac{3}{2}\frac{(p^{\prime
\prime})^{2}}{p^{\prime2}}\right)  v.
\]
In Section 5 we shall discuss the derivative of $V$ and we apply the results
to the function $f_{n}$ defined in (\ref{eqdeffn}). We shall show that
$V\left(  f_{n}\right)  $ is a positive function on $\left(  0,\infty\right)
$, and from this we deduce our main result, namely that $w\left(
f_{n}\right)  $ is positive $\left(  0,j_{1}\left(  f_{n}\right)  \right)  $
and increasing on $\left(  0,j_{1}\left(  f_{n}^{\prime}\right)  \right)  .$
Here $j_{1}\left(  f\right)  $ denotes the first positive zero of a function
$f$ defined on $\left(  0,\infty\right)  .$

\section{The function $v\left(  f\right)  $}

In this section we present results which are needed in the next sections and
generalize some results in \cite{CMP23}. Throughout the paper we use the
notation%
\[
v\left(  f\right)  =f^{\prime}f^{\prime}-f^{\prime\prime}f
\]
and we simply write $v$ instead of $v\left(  f\right)  $ if there is no
ambiguity. Further $v^{\prime}\left(  f\right)  $ will denote the derivative
of $v\left(  f\right)  .$

Note that for $\alpha\in\mathbb{R}$ we have
\[
v\left(  x^{\alpha}\right)  =\alpha^{2}x^{2\alpha-2}-\alpha\left(
\alpha-1\right)  x^{2\alpha-2}=\alpha x^{2\alpha-2}.
\]
Hence $v\left(  x^{\alpha}\right)  \geq0$ for all $\alpha\geq0$ and $x>0.$

The following result is straightforward:

\begin{proposition}
\label{Prop0}Let $f,g\in C^{2}\left(  a,b\right)  .$ Then
\[
v\left(  fg\right)  =v\left(  f\right)  g^{2}+f^{2}v\left(  g\right)  .
\]
In particular $v\left(  f\right)  \left(  x\right)  \geq0$ and $v\left(
g\right)  \left(  x\right)  \geq0$ for all $x\in\left(  a,b\right)  $ implies
that $v\left(  fg\right)  \left(  x\right)  \geq0$ for all $x\in\left(
a,b\right)  .$
\end{proposition}

\begin{proposition}
\label{Prop1}Let $p,q\in C^{1}\left(  a,b\right)  .$ If $f$ is a solution of
$f^{\prime\prime}+pf^{\prime}+qf=0$ then%
\begin{equation}
v^{\prime}\left(  f\right)  +pv\left(  f\right)  =p^{\prime}f^{\prime
}f+q^{\prime}f^{2}.\label{eqvde}%
\end{equation}

\end{proposition}

\begin{proof}
We make the substitution $f^{\prime\prime}=-pf^{\prime}-qf$ in the definition
of $v,$ so
\begin{equation}
v=\left(  f^{\prime}\right)  ^{2}-f^{\prime\prime}f=\left(  f^{\prime}\right)
^{2}+pf^{\prime}f+qf^{2}.\label{eq1}%
\end{equation}
Differentiation of the right hand side gives
\[
v^{\prime}=2f^{\prime}f^{\prime\prime}+p^{\prime}f^{\prime}f+pf^{\prime\prime
}f+pf^{\prime}f^{\prime}+q^{\prime}f^{2}+2qf^{\prime}f.
\]
Replace $f^{\prime\prime}$ by $-pf^{\prime}-qf$, so
\begin{align*}
v^{\prime}  & =\ -2pf^{\prime}f^{\prime}-2qf^{\prime}f+p^{\prime}f^{\prime
}f-p^{2}f^{\prime}f-pqf^{2}+pf^{\prime}f^{\prime}+q^{\prime}f^{2}+2qf^{\prime
}f\\
& =-pf^{\prime}f^{\prime}+\left(  p^{\prime}-p^{2}\right)  f^{\prime}f+\left(
q^{\prime}-pq\right)  f^{2}.
\end{align*}
Using (\ref{eq1}) we obtain our claim that $v^{\prime}=-pv+p^{\prime}%
f^{\prime}f+q^{\prime}f^{2}.$
\end{proof}

\begin{theorem}
Assume that $f,p,q\in C^{2}\left[  a,b\right]  $ such that $f^{\prime\prime
}+pf^{\prime}+q=0$ and $f\left(  a\right)  =f^{\prime}\left(  a\right)  =0$
and define $P=\int p\left(  x\right)  dx$. Then
\begin{equation}
v\left(  f\right)  \left(  x\right)  =\frac{1}{2}f^{2}\left(  x\right)
p^{\prime}\left(  x\right)  -\frac{1}{2}e^{-P\left(  x\right)  }\int_{a}%
^{x}f\left(  t\right)  ^{2}\left(  p^{\prime\prime}+p^{\prime}p-2q^{\prime
}\right)  e^{P}dt\label{eqnewff}%
\end{equation}

\end{theorem}

\begin{proof}
We write $v$ instead of $v\left(  f\right)  .$ Multiply (\ref{eqvde}) with
$e^{P},$ so
\[
\left(  v^{\prime}+pv\right)  e^{P}=p^{\prime}f^{\prime}fe^{P}+q^{\prime}%
f^{2}e^{P}.
\]
Since $P^{\prime}=p$ we have
\[
\frac{d}{dx}\left(  e^{P}v\right)  =e^{P}v^{\prime}+e^{P}P^{\prime}%
v=e^{P}\left(  v^{\prime}+pv\right)  =p^{\prime}f^{\prime}fe^{P}+q^{\prime
}f^{2}e^{P}.
\]
Note that $v\left(  a\right)  =0$ since $f^{\prime}\left(  a\right)  =f\left(
a\right)  =0.$ Integration gives
\[
e^{P\left(  x\right)  }v\left(  x\right)  =\int_{a}^{x}\frac{d}{dt}\left(
e^{P}v\right)  dt=\int_{a}^{x}p^{\prime}f^{\prime}fe^{P}dt+\int_{a}%
^{x}q^{\prime}f^{2}e^{P}dt.
\]
We apply partial integration to the first integral of the right hand side:
\[
\int_{a}^{x}p^{\prime}f^{\prime}fe^{P}dt=\frac{1}{2}f^{2}\left(  x\right)
p^{\prime}\left(  x\right)  e^{P\left(  x\right)  }-\int_{a}^{x}\frac{1}%
{2}f^{2}\left(  p^{\prime\prime}+p^{\prime}p\right)  e^{P}dt.
\]
Thus
\[
v\left(  x\right)  =\frac{1}{2}f^{2}\left(  x\right)  p^{\prime}\left(
x\right)  -e^{-P\left(  x\right)  }\int_{a}^{x}\frac{1}{2}f^{2}\left(
p^{\prime\prime}+p^{\prime}p-2q^{\prime}\right)  e^{P}dt.
\]

\end{proof}

Formula (\ref{eqnewff}) allows us to give a short proof of an important result
in \cite{CMP23}, namely the positivity of $v\left(  f_{n}\right)  \left(
x\right)  $ for $x>0$. The proof in \cite{CMP23} depends on recurrence
relations of spherical Bessel functions. Further properties of $v\left(
f_{n}\right)  $ obtained in \cite{CMP23} will be reproven in Corollary
\ref{Cor2} below.

\begin{theorem}
For the function $f_{n}$ defined in (\ref{eqdeffn}) for $n\in\mathbb{N}_{0}$
the following identity
\[
v\left(  f_{n}\right)  \left(  x\right)  =\frac{n}{x^{2}}f_{n}^{2}\left(
x\right)  +2n\left(  n+1\right)  x^{2n}\int_{0}^{x}\frac{f_{n}^{2}\left(
t\right)  }{t^{2n+3}}dt
\]
holds and $v\left(  f_{n}\right)  \left(  x\right)  >0$ for all $x>0.$
\end{theorem}

\begin{proof}
We use formula (\ref{eqnewff}) and (\ref{eqdefnewf}): we set $p=-\frac{2n}%
{x},$ then $P=-2n\ln x$ and $e^{P}=e^{-2n\ln x}=x^{-2n}.$ Further
\[
p^{\prime\prime}+p^{\prime}p-2q^{\prime}=-\frac{4n}{x^{3}}+\frac{2n}{x^{2}%
}\left(  -\frac{2n}{x}\right)  =-\frac{4n+4n^{2}}{x^{3}}=-\frac{4n\left(
n+1\right)  }{x^{3}}.
\]

\end{proof}

\begin{theorem}
For the Bessel function $J_{\nu}$ the following identity holds for $\nu>1$:
\begin{equation}
v\left(  J_{\nu}\right)  \left(  x\right)  =-\frac{1}{2x^{2}}J_{\nu}\left(
x\right)  ^{2}+\frac{4\nu^{2}-1}{2x}\int_{0}^{x}J_{\nu}\left(  t\right)
^{2}\frac{1}{t^{2}}dt\label{eqJW}%
\end{equation}

\end{theorem}

\begin{proof}
We know that $J_{\nu}\left(  0\right)  =J_{\nu}^{\prime}\left(  0\right)  =0$
since $\nu>1.$ Further $J_{\nu}$ satisfies the differential equation%
\begin{equation}
J_{\nu}^{\prime\prime}+\frac{1}{x}J_{\nu}^{\prime}+\left(  1-\frac{\nu^{2}%
}{x^{2}}\right)  J_{\nu}=0.\label{eqJ}%
\end{equation}
Then $p\left(  x\right)  =\frac{1}{x}$ and $p^{\prime}\left(  x\right)
=-\frac{1}{x^{2}}.$ Further $P\left(  x\right)  =\ln x,$ so $e^{-P\left(
x\right)  }=\frac{1}{x}$ and
\[
p^{\prime\prime}+p^{\prime}p-2q^{\prime}=\frac{2}{x^{3}}-\frac{1}{x^{3}%
}-2\left(  \frac{2\nu^{2}}{x^{3}}\right)  =-\frac{1}{x^{3}}\left(  4\nu
^{2}-1\right)  .
\]

\end{proof}

It is a natural to ask whether $v\left(  J_{\nu}\right)  $ is a positive
function for $x>0$. Unfortunately the summands in (\ref{eqJW}) have different
signs for $\nu>1.$ We now present a general criterion for positivity of
$v\left(  f\right)  $:

\begin{theorem}
\label{ThmMain1} Let $p,q\in C^{1}\left(  a,b\right)  $ and assume that the
derivatives $p^{\prime}$ and $q^{\prime}$ do not have a common zero on
$\left(  a.b\right)  $. Let $f$ be a solution of $f^{\prime\prime}+pf^{\prime
}+qf=0$ in $\left(  a,b\right)  $ and assume that $v=f^{\prime}f^{\prime
}-ff^{\prime\prime}$ extends continuously to $\left[  a,b\right]  $ with
$v\left(  a\right)  \geq0$ and $v\left(  b\right)  \geq0.$ If
\[
q\left(  x\right)  -\frac{p\left(  x\right)  }{p^{\prime}\left(  x\right)
}q^{\prime}\left(  x\right)  \ \geq0
\]
for all $x\in\left(  a,b\right)  $ then $v\left(  x\right)  \geq0$ for all
$\left[  a,b\right]  .$
\end{theorem}

\begin{proof}
Let $P\left(  x\right)  $ be an anti-derivative of $p\left(  x\right)  .$ We
consider the function $V\left(  x\right)  =e^{P\left(  x\right)  }v\left(
x\right)  $. Our assumptions imply that $V\left(  a\right)  \geq0$ and
$V\left(  b\right)  \geq0.$ Then
\begin{equation}
e^{-P\left(  x\right)  }V^{\prime}\left(  x\right)  =e^{-P\left(  x\right)
}\frac{d}{dx}\left(  e^{P\left(  x\right)  }v\left(  x\right)  \right)
=v^{\prime}\left(  x\right)  +p\left(  x\right)  v\left(  x\right)
\label{eqVV}%
\end{equation}
for all $x\in\left(  a,b\right)  .$ Let $x_{0}\in\left(  a,b\right)  $ be a
critical point of $V\left(  x\right)  .$ We will show that $V\left(
x_{0}\right)  \geq0.$ Formula (\ref{eqVV}) and Proposition \ref{Prop1} show
that
\[
0=v^{\prime}\left(  x_{0}\right)  +p\left(  x_{0}\right)  v\left(
x_{0}\right)  =p^{\prime}\left(  x_{0}\right)  f^{\prime}\left(  x_{0}\right)
f\left(  x_{0}\right)  +q^{\prime}\left(  x_{0}\right)  f^{2}\left(
x_{0}\right)  .
\]
If $p^{\prime}\left(  x_{0}\right)  =0$ we see that $q^{\prime}\left(
x_{0}\right)  f^{2}\left(  x_{0}\right)  =0.$ Since $p^{\prime}$ and
$q^{\prime}$ do have a common zero we infer that $q^{\prime}\left(
x_{0}\right)  \neq0$ and therefore $f\left(  x_{0}\right)  =0.$ Then $v\left(
x_{0}\right)  =f^{\prime}\left(  x_{0}\right)  ^{2}$ and $V\left(
x_{0}\right)  \geq0.$ Now assume that $p^{\prime}\left(  x_{0}\right)  \neq0.$
Then
\[
f^{\prime}\left(  x_{0}\right)  f\left(  x_{0}\right)  =-\frac{q^{\prime
}\left(  x_{0}\right)  }{p^{\prime}\left(  x_{0}\right)  }f^{2}\left(
x_{0}\right)  .
\]
Multiply the equation $f^{\prime\prime}+pf^{\prime}+qf=0$ with $f.$ Then
\[
0=f^{\prime\prime}\left(  x_{0}\right)  f\left(  x_{0}\right)  +\left(
q\left(  x_{0}\right)  -\frac{p\left(  x_{0}\right)  }{p^{\prime}\left(
x_{0}\right)  }q^{\prime}\left(  x_{0}\right)  \right)  f^{2}\left(
x_{0}\right)
\]
and
\begin{align*}
v\left(  x_{0}\right)   & =f^{\prime}\left(  x_{0}\right)  f^{\prime}\left(
x_{0}\right)  -f^{\prime\prime}\left(  x_{0}\right)  f\left(  x_{0}\right) \\
& =f^{\prime}\left(  x_{0}\right)  f^{\prime}\left(  x_{0}\right)  +\left(
q\left(  x_{0}\right)  -\frac{p\left(  x_{0}\right)  }{p^{\prime}\left(
x_{0}\right)  }q^{\prime}\left(  x_{0}\right)  \right)  f^{2}\left(
x_{0}\right)  .
\end{align*}
By assumption $q-\frac{p}{p^{\prime}}q^{\prime}\geq0.$ It follows that
$V\left(  x_{0}\right)  \geq0$ for any critical point of $V,$ and $V$ is
non-negative on the endpoints of the interval. So $V\left(  x\right)  \geq0$
and $v\left(  x\right)  \geq0$ for all $x\in\left[  a,b\right]  .$
\end{proof}

\begin{corollary}
Let $\nu$ be a real non-negative number. Then $v\left(  J_{\nu}\right)
\left(  x\right)  \geq0$ for all $x\geq0.$
\end{corollary}

\begin{proof}
Since $J_{\nu}$ satisfies the differential equation (\ref{eqJ}) we see that
\[
q\left(  x\right)  -\frac{p\left(  x\right)  }{p^{\prime}\left(  x\right)
}q^{\prime}\left(  x\right)  =\frac{x^{2}+\nu^{2}}{x^{2}}.
\]
Now we apply Theorem \ref{ThmMain1} to the interval $\left(  0,j_{\nu
,k}\right)  $ where $j_{\nu,k}$ is the $k$-th positive zero of $J_{\nu}.$ Then
$v\left(  J_{\nu}\right)  \left(  0\right)  =0$ since $\nu\geq0$ and $v\left(
J_{\nu}\right)  \left(  j_{\nu,k}\right)  =J_{\nu}\left(  j_{\nu,k}\right)
^{2}\geq0.$
\end{proof}

The following figure shows that $v\left(  J_{3.4}\right)  $ is not an
increasing function:
%TCIMACRO{\FRAME{dtbpFUX}{3.0113in}{1.5056in}{0pt}{\Qcb{ Graph of $v\left(
%J_{3.4}\right)  $ }}{}{Plot}{\special{ language "Scientific Word";
%type "MAPLEPLOT";  width 3.0113in;  height 1.5056in;  depth 0pt;
%display "USEDEF";  plot_snapshots TRUE;  mustRecompute FALSE;
%lastEngine "MuPAD";  xmin "0";  xmax "10";  xviewmin "-0.000999904659828";
%xviewmax "10.0010000009906";  yviewmin "-8.99659484740551E-6";
%yviewmax "0.089974855085051";  plottype 4;
%axesFont "Times New Roman,12,0000000000,useDefault,normal";  numpoints 100;
%plotstyle "patch";  axesstyle "normal";  axestips FALSE;  xis \TEXUX{x};
%var1name \TEXUX{$x$};
%function \TEXUX{$-\det \left( \MATRIX{2,2}{c}\VR{,,c,,,}{,,c,,,}{,,,,,}\HR{,,}\CELL{\left( J_{3.5}\left( x\right) \right) }\CELL{\frac{d}{dx}\left( J_{3.5}\left( x\right) \right) }\CELL{\frac{d}{dx}\left( J_{3.5}\left( x\right) \right) }\CELL{\frac{d^{2}}{dx^{2}}\left( J_{3.5}\left( x\right) \right) }\right) $};
%linecolor "black";  linestyle 1;  pointstyle "point";  linethickness 1;
%lineAttributes "Solid";  var1range "0,10";  num-x-gridlines 100;
%curveColor "[flat::RGB:0000000000]";  curveStyle "Line";  rangeset "X";
%VCamFile 'SUEMC806.xvz';valid_file "T";
%tempfilename 'SUEMC803.wmf';tempfile-properties "PR";}} }%
%BeginExpansion
\begin{center}
\fbox{\includegraphics[
height=1.5056in,
width=3.0113in
]
{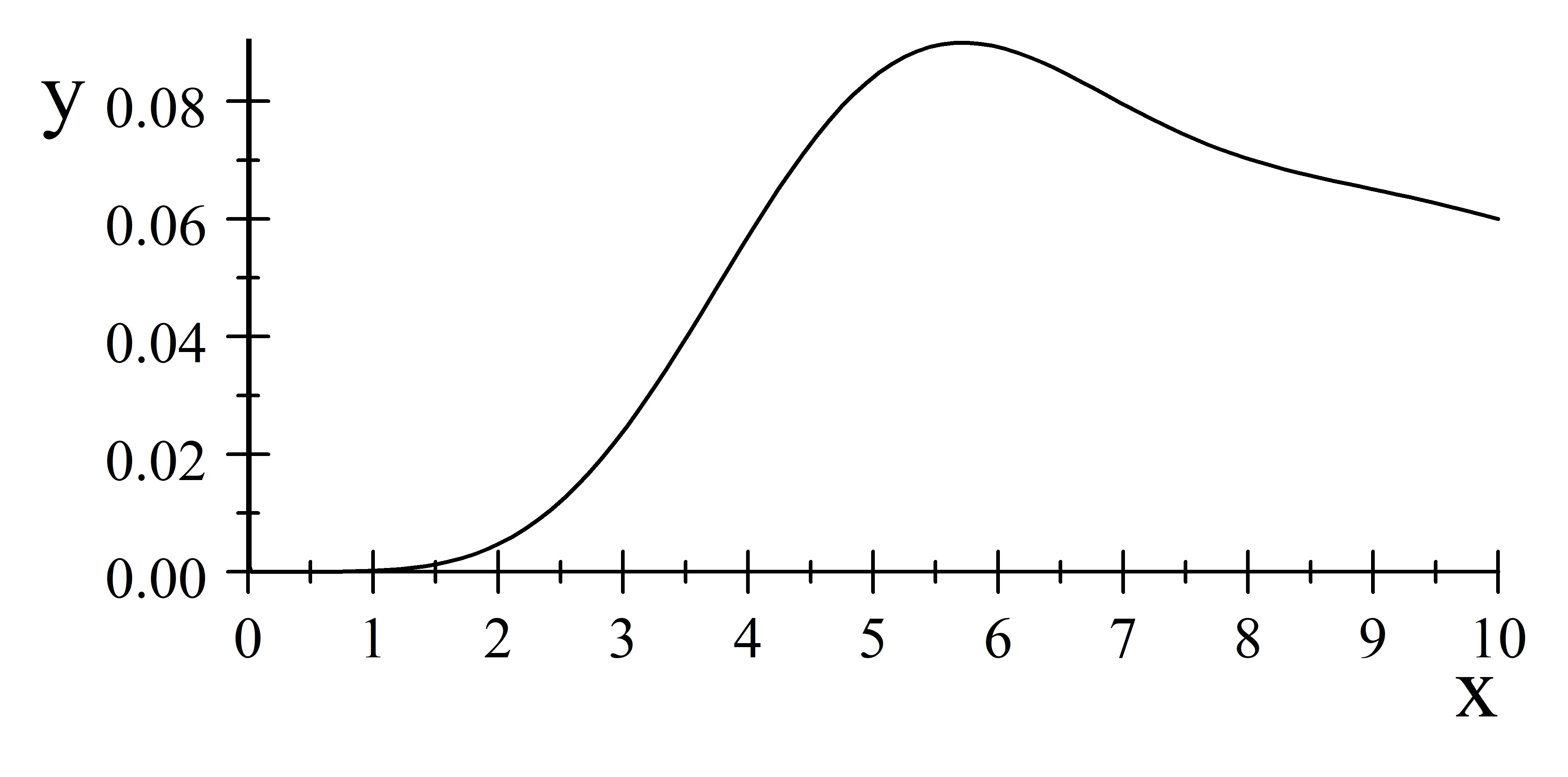}
}\\
Graph of $v\left(  J_{3.4}\right)  $
\end{center}
%EndExpansion

Next we turn to the question under which conditions $v\left(  f\right)  $ is increasing.

\begin{proposition}
\label{Prop2}Let $p,q\in C^{2}\left(  a,b\right)  .$ If $f$ is a solution of
$f^{\prime\prime}+pf^{\prime}+qf=0$ then
\begin{equation}
v^{\prime\prime}+\left(  p-\frac{p^{\prime\prime}}{p^{\prime}}\right)
v^{\prime}+\left(  2p^{\prime}-\frac{p^{\prime\prime}p}{p^{\prime}}\right)
v=2p^{\prime}f^{\prime}f^{\prime}+f^{2}p^{\prime}\frac{d}{dx}\frac{q^{\prime}%
}{p^{\prime}}+2q^{\prime}f^{\prime}f\ \label{eqLHS1}%
\end{equation}
for all $x\in\left(  a,b\right)  $ with $p^{\prime}\left(  x\right)  \neq0.$
\end{proposition}

\begin{proof}
Proposition \ref{Prop1} yields $v^{\prime}+pv=p^{\prime}f^{\prime}f+q^{\prime
}f^{2}$ and it follows that
\[
\frac{d}{dx}\frac{v^{\prime}+pv}{p^{\prime}}=\frac{d}{dx}\left(  f^{\prime
}f+\frac{q^{\prime}}{p^{\prime}}f^{2}\right)  =f^{\prime\prime}f+f^{\prime
}f^{\prime}+f^{2}\frac{d}{dx}\frac{q^{\prime}}{p^{\prime}}+2\frac{q^{\prime}%
}{p^{\prime}}f^{\prime}f.
\]
Since $f^{\prime\prime}f+f^{\prime}f^{\prime}=2f^{\prime}f^{\prime}-v$ we
obtain
\begin{equation}
\frac{d}{dx}\frac{v^{\prime}+pv}{p^{\prime}}+v=2f^{\prime}f^{\prime}%
+f^{2}\frac{d}{dx}\frac{q^{\prime}}{p^{\prime}}+2\frac{q^{\prime}}{p^{\prime}%
}f^{\prime}f.\label{eqLHS2}%
\end{equation}
It is easy to see that
\[
p^{\prime}\left(  \frac{d}{dx}\frac{v^{\prime}+pv}{p^{\prime}}+v\right)
=\left(  v^{\prime\prime}+p^{\prime}v+pv^{\prime}\right)  -\left(  v^{\prime
}+pv\right)  \frac{p^{\prime\prime}}{p^{\prime}}+p^{\prime}v
\]
which is equal to the left hand side in (\ref{eqLHS1}). The proof is complete.
\end{proof}

\begin{corollary}
\label{Cor2}Let $f_{n}\left(  x\right)  $ be defined as in (\ref{eqdeffn}).
Then $v\left(  f_{n}\right)  $ and $v^{\prime}\left(  f_{n}\right)  $ are
positive and increasing on $\left(  0,\infty\right)  $. Further $v^{\prime
\prime}\left(  f_{n}\right)  \geq0$ on $\left(  0,\infty\right)  .$
\end{corollary}

\begin{proof}
We apply Proposition \ref{Prop2} with $p=-2n/x.$ Observe that $2p^{\prime
}-\frac{p^{\prime\prime}p}{p^{\prime}}=0.$ Further
\[
p-\frac{p^{\prime\prime}}{p^{\prime}}=\left(  \frac{-2n}{x}\right)
-\frac{\frac{d^{2}}{dx^{2}}\left(  \frac{-2n}{x}\right)  }{\frac{d}{dx}\left(
\frac{-2n}{x}\right)  }=-\frac{2\left(  n-1\right)  }{x}.
\]
It follows that $v=v\left(  f_{n}\right)  $ satisfies the equation
\[
v^{\prime\prime}-\frac{2\left(  n-1\right)  }{x}v^{\prime}=\frac{2n}{x^{2}%
}f^{\prime}f^{\prime}.
\]
It is easy to see that
\[
x^{2n-2}\frac{d}{dx}\left(  x^{2-2n}v^{\prime}\right)  =v^{\prime\prime}%
-\frac{2\left(  n-1\right)  }{x}v^{\prime}=\frac{2n}{x^{2}}\left(  f^{\prime
}\right)  ^{2}\geq0.
\]
It follows that $x^{2-2n}v^{\prime}\left(  x\right)  $ is increasing. Since
$v^{\prime}\left(  0\right)  =0$ we conclude that $v^{\prime}\geq0$ on
$\left(  0,\infty\right)  .$ Further $v^{\prime}\left(  x\right)  $ is
increasing since it is a product two positive increasing functions $x^{2n-2}$
and $x^{2-2n}v^{\prime}\left(  x\right)  .$ Since $v^{\prime}\left(  0\right)
=0$ it follows that $v^{\prime}\geq0$ on $\left(  0,\infty\right)  .$ It
follows that $v$ is increasing on $\left(  0,\infty\right)  .$ Finally we see
that
\[
v^{\prime\prime}\left(  x\right)  =\frac{2n}{x^{2}}f^{\prime}\left(  x\right)
^{2}+\frac{2\left(  n-1\right)  }{x}v^{\prime}\left(  x\right)
\]
is non-negative since we have shown that $v^{\prime}\left(  x\right)  \geq0$
for all $x>0.$
\end{proof}

As mentioned in \cite[p. 4]{CMP23} one has%
\[
f_{0}=\sin x,\qquad f_{1}\left(  x\right)  =\sin x-x\cos x,\qquad
f_{2}=\left(  3-x^{2}\right)  \sin x-3x\cos x.
\]
It is easy to see that
\[
v\left(  f_{0}\right)  \left(  x\right)  =\cos^{2}x+\sin^{2}x=1\text{ and
}v\left(  f_{1}\right)  \left(  x\right)  =\frac{1}{2}\cos2x+x^{2}-\frac{1}%
{2}.
\]
Thus all derivatives of $v\left(  f_{0}\right)  $ are non-negative (and zero).
Further $v\left(  f_{1}\right)  $ is positive and increasing on $\left(
0,\infty\right)  $, and the same is true for $v^{\prime}\left(  x\right)
=2x-\sin2x$. Further $v^{\prime\prime}\left(  x\right)  =2-2\cos2x\geq0$ but
$v^{\prime\prime}$ is not increasing since $v^{\prime\prime\prime}\left(
x\right)  =4\sin2x$ changes its sign -- so higher derivatives of $v$ of order
$\geq3$ do not have a nice monotonicity properties.

We conclude this section with the following result:

\begin{proposition}
Let $p,q\in C^{1}\left[  a,b\right]  $ and $f$ a solution of $f^{\prime\prime
}+pf^{\prime}+qf=0.$ Then
\[
v\left(  f^{\prime}\right)  =p^{\prime}f^{2}+q^{\prime}f^{\prime}f+qv\left(
f\right)  .
\]
In particular, if $q=1$ and $p^{\prime}\geq0$ then $v\left(  f^{\prime
}\right)  \geq v\left(  f\right)  .$
\end{proposition}

\begin{proof}
Since $v\left(  f^{\prime}\right)  =f^{\prime\prime}f^{\prime\prime}%
-f^{\prime}f^{\prime\prime\prime}$ and $\frac{d}{dx}\left(  f^{\prime\prime
}+pf^{\prime}+qf\right)  =0$ we see that
\begin{align*}
v\left(  f^{\prime}\right)   & =\left(  pf^{\prime}+qf\right)  ^{2}-f^{\prime
}\left(  -p^{\prime}f^{\prime}-pf^{\prime\prime}-qf^{\prime}-q^{\prime
}f\right) \\
& =p^{2}f^{\prime2}+2pqf^{\prime}f+q^{2}f^{2}+p^{\prime}f^{\prime2}%
+pf^{\prime}\left(  -pf^{\prime}-qf\right)  +qf^{\prime}f^{\prime}+q^{\prime
}f^{\prime}f\\
& =\left(  p^{\prime}+q\right)  f^{\prime2}+\left(  pq+q^{\prime}\right)
f^{\prime}f+q^{2}f^{2}.
\end{align*}
$\ $Since $v\left(  f\right)  =qf^{2}+pf^{\prime}f+f^{\prime2}$ we can write
\begin{align*}
v\left(  f^{\prime}\right)   & =\left(  q+p^{\prime}\right)  f^{\prime
2}+\left(  pq+q^{\prime}\right)  f^{\prime}f+q\left(  v\left(  f\right)
-f^{\prime}f^{\prime}-pf^{\prime}f\right) \\
& =p^{\prime}f^{\prime2}+q^{\prime}f^{\prime}f+qv\left(  f\right)  .
\end{align*}
The proof is complete.
\end{proof}

\section{The function $w\left(  f\right)  $}

In this section we want to compute the Wronski-determinant of the system
$f,f^{\prime},f^{\prime\prime}.$ The basic idea is to replace all derivatives
$f^{\left(  k\right)  }$ in the Wronski matrix with $k\geq2$ with expressions
which depends only on $f,f^{\prime}$ using the assumption that $f$ is a
solution of the differential equation $f^{\prime\prime}+pf^{\prime}+qf=0.$ We
begin with the following result:

\begin{theorem}
\label{ThmMain2}Let $f,p,q\in C^{2}\left(  \left[  a,b\right]  ,\mathbb{C}%
\right)  $ and $f$ be a solution of the differential equation $f^{\prime
\prime}+pf^{\prime}+qf=0$. Then
\begin{equation}
w:=\det\left(
\begin{array}
[c]{ccc}%
f^{\prime\prime} & f^{\prime} & f\\
f^{\prime\prime\prime} & f^{\prime\prime} & f^{\prime}\\
f^{\left(  4\right)  } & f^{\prime\prime\prime} & f^{\prime\prime}%
\end{array}
\right)  =\left(  p^{\prime}f^{\prime}+q^{\prime}f\right)  ^{2}f-A\cdot
v\label{eqw1}%
\end{equation}
where $v=f^{\prime}f^{\prime}-f^{\prime\prime}f$, and $A$ is given by the
equivalent formulae:
\begin{align*}
A  & =2p^{\prime}f^{\prime\prime}+\left(  p^{\prime\prime}+p^{\prime
}p+2q^{\prime}\right)  f^{\prime}+\left(  q^{\prime\prime}+pq^{\prime}\right)
f\\
A  & =\left(  p^{\prime\prime}-p^{\prime}p+2q^{\prime}\right)  f^{\prime
}+\left(  -2p^{\prime}q+q^{\prime\prime}+pq^{\prime}\right)  f.
\end{align*}

\end{theorem}

\begin{remark}
The identity (\ref{eqw1}) can be used for evaluating $w$ at special points:
e.g. for any zero $x_{0}$ of $f$ the following identity holds:
\[
w\left(  x_{0}\right)  =-\left[  p^{\prime\prime}\left(  x_{0}\right)
-p^{\prime}\left(  x_{0}\right)  p\left(  x_{0}\right)  +2q^{\prime}\left(
x_{0}\right)  \right]  \cdot f^{\prime3}\left(  x_{0}\right)  .
\]

\end{remark}

\begin{proof}
We differentiate the equation $f^{\prime\prime}+pf^{\prime}+qf=0$ and we
obtain
\begin{align*}
f^{\prime\prime\prime}+pf^{\prime\prime}+\left(  p^{\prime}+q\right)
f^{\prime}+q^{\prime}f  & =0\\
f^{\prime\prime\prime\prime}+pf^{\prime\prime\prime}+\left(  2p^{\prime
}+q\right)  f^{\prime\prime}+\left(  p^{\prime\prime}+2q^{\prime}\right)
f^{\prime}+q^{\prime\prime}f  & =0.
\end{align*}
In the Wronski matrix we add to the first column the $p$ multiple of column 2
and the $q$ multiple of column 3 arriving at
\[
w=\det\left(
\begin{array}
[c]{ccc}%
f^{\prime\prime} & f^{\prime} & f\\
f^{\left(  3\right)  } & f^{\prime\prime} & f^{\prime}\\
f^{\left(  4\right)  } & f^{\left(  3\right)  } & f^{\prime\prime}%
\end{array}
\right)  =\det\left(
\begin{array}
[c]{ccc}%
f^{\prime\prime}+pf^{\prime}+qf & f^{\prime} & f\\
f^{\left(  3\right)  }+pf^{\prime\prime}+qf^{\prime} & f^{\prime\prime} &
f^{\prime}\\
f^{\left(  4\right)  }+pf^{\left(  3\right)  }+qf^{\prime\prime} & f^{\left(
3\right)  } & f^{\prime\prime}%
\end{array}
\right)  .
\]
Using the above differential expressions we see that%
\begin{equation}
w=\det\left(
\begin{array}
[c]{ccc}%
0 & f^{\prime} & f\\
-p^{\prime}f^{\prime}-q^{\prime}f & f^{\prime\prime} & f^{\prime}\\
-2p^{\prime}f^{\prime\prime}-\left(  p^{\prime\prime}+2q^{\prime}\right)
f^{\prime}-q^{\prime\prime}f & f^{\left(  3\right)  } & f^{\prime\prime}%
\end{array}
\right)  .\label{eqExpan}%
\end{equation}
Recall that
\[
v^{\prime}=\frac{d}{dx}\left(  f^{\prime}f^{\prime}-f^{\prime\prime}f\right)
=f^{\prime\prime}f^{\prime}-f^{\prime\prime\prime}f=\det\left(
\begin{array}
[c]{cc}%
f^{\prime} & f\\
f^{\left(  3\right)  } & f^{\prime\prime}%
\end{array}
\right)  .
\]
Expansion at the first column in (\ref{eqExpan}) gives
\[
w=\left(  p^{\prime}f^{\prime}+q^{\prime}f\right)  v^{\prime}-\left(
2p^{\prime}f^{\prime\prime}+\left(  p^{\prime\prime}+2q^{\prime}\right)
f^{\prime}+q^{\prime\prime}f\right)  v.
\]
By Proposition \ref{Prop1}, $v^{\prime}+pv=p^{\prime}f^{\prime}f+q^{\prime
}f^{2}=\left(  p^{\prime}f^{\prime}+q^{\prime}f\right)  f$, and we obtain
\[
w=\left(  p^{\prime}f^{\prime}+q^{\prime}f\right)  ^{2}f-\left(  2p^{\prime
}f^{\prime\prime}+\left(  p^{\prime\prime}+2q^{\prime}\right)  f^{\prime
}+q^{\prime\prime}f+\left(  p^{\prime}f^{\prime}+q^{\prime}f\right)  p\right)
v.
\]
Now replace $f^{\prime\prime}=-pf^{\prime}-qf$ and the result follows.
\end{proof}

\begin{corollary}
Let $f$ be a solution of the differential equation $f^{\prime\prime
}+pf^{\prime}+qf=0.$ Then
\[
w=a_{0}f^{3}+a_{1}f^{\prime}f^{2}+a_{2}f^{\prime2}f+a_{3}f^{\prime3}%
\]
with
\begin{align*}
a_{0}  & =2p^{\prime}q^{2}-pqq^{\prime}+(q^{\prime})^{2}-q^{\prime\prime}q,\\
a_{1}  & =3p^{\prime}pq-p^{\prime\prime}q-2qq^{\prime}-p^{2}q^{\prime
}+2p^{\prime}q^{\prime}-pq^{\prime\prime},\\
a_{2}  & =(p^{\prime})^{2}+p^{2}p^{\prime}-pp^{\prime\prime}+2p^{\prime
}q-3pq^{\prime}-q^{\prime\prime},\\
a_{3}  & =pp^{\prime}-p^{\prime\prime}-2q^{\prime}.
\end{align*}

\end{corollary}

\begin{proof}
Put $x=f$ and $f^{\prime}=y$ and write $v=qx^{2}+pxy+y^{2}.$ Then
\[
w=\left(  p^{\prime}y+q^{\prime}x\right)  ^{2}x-A\left(  qx^{2}+pxy+y^{2}%
\right)
\]
where
\[
A=\left(  2p^{\prime}\left(  -py-qx\right)  +\left(  p^{\prime\prime
}+2q^{\prime}\right)  y+q^{\prime\prime}x+\left(  p^{\prime}y+q^{\prime
}x\right)  p\right)
\]
Straightforward calculations (using e.g. Maple) gives the result.
\end{proof}

Now let $p=\left(  -2n\right)  /x$ and $q=1.$ Then
\begin{align*}
a_{0}  & =4nx^{-2}\text{ and }a_{1}=-4nx^{-3}\left(  3n-1\right) \\
a_{2}  & =4nx^{-2}\left(  2n^{2}x^{-2}-nx^{-2}+1\right)  \text{ and }%
a_{3}=-4nx^{-3}\left(  n-1\right)
\end{align*}
and we obtain:

\begin{corollary}
\label{Cor5}Let $f_{n}\left(  x\right)  $ be defined as in (\ref{eqdeffn}).
Then
\begin{equation}
w\left(  f_{n}\right)  =\frac{4n}{x^{2}}\left(  f_{n}^{3}-\frac{3n-1}{x}%
f_{n}^{\prime}f_{n}^{2}+\frac{2n^{2}-n+x^{2}}{x^{2}}f_{n}^{\prime2}f_{n}%
-\frac{n-1}{x}f_{n}^{\prime3}\right)  .\label{eqs4}%
\end{equation}

\end{corollary}

Let us recall that $j_{\nu,k}$ and $j_{\nu,k}^{\prime}$ respectively denotes
the $k$-th positive zero of $J_{\nu}$ and $J_{\nu}^{\prime}$ respectively. It
is well known (see \cite[p. 479]{Watson}) that the zeros of $J_{\nu}$ and
$J_{\nu+1}$ interlace, so we have
\begin{equation}
0<j_{\nu,1}<j_{\nu+1,1}<j_{\nu,2}.\label{eqjj}%
\end{equation}
Now we consider the zeros of $f_{n}$ and $f_{n}^{\prime}$. We define
$j_{k}\left(  f_{n}\right)  $ as the $k$-th positive zero of $f,$ and
$j_{k}\left(  f_{n}^{\prime}\right)  $ as the $k$-th positive zero of
$f_{n}^{\prime}.$ From the definition (\ref{eqdeffn}) it is clear that
$j_{k}\left(  f_{n}\right)  =j_{n+\frac{1}{2},k}$. Proposition 3 in
\cite[Proposition 3]{CMP23} says that $f_{n}^{\prime}\left(  x\right)
=xf_{n-1}\left(  x\right)  $ which shows that $j_{k}\left(  f_{n}^{\prime
}\right)  =j_{n-\frac{1}{2},k}.$ Thus with $\nu=n-\frac{1}{2}$ we obtain from
(\ref{eqjj}) that
\[
0<j_{1}\left(  f_{n}^{\prime}\right)  <j_{1}\left(  f_{n}\right)
<j_{2}\left(  f_{n}^{\prime}\right)
\]
so $f\,_{n}^{\prime}\left(  x\right)  <0$ for all $x\in\left(  j_{1}\left(
f_{n}^{\prime}\right)  ,j_{1}\left(  f_{n}\right)  \right)  .$ Since
$J_{n+\frac{1}{2}}$ is increasing and positive on $\left(  0,j_{n+\frac{1}%
{2},1}^{\prime}\right)  $ we see that $f_{n}$ defined in (\ref{eqdeffn}) is,
as the product two increasing positive functions, increasing and positive on
$\left(  0,j_{n+\frac{1}{2},1}^{\prime}\right)  .$ Thus we obtain%
\begin{equation}
j_{1}\left(  f_{n}^{\prime}\right)  \geq j_{n+\frac{1}{2},1}^{\prime}%
>n+\frac{1}{2}\label{eqnnn}%
\end{equation}
where we used that $j_{\nu}^{\prime}>\nu$, see \cite[p. 485]{Watson}.

\begin{corollary}
\label{Cor4}The function $w\left(  f_{n}\right)  $ is positive on the interval
$\left(  j_{1}\left(  f_{n}^{\prime}\right)  ,\,j_{1}\left(  f_{n}\right)
\right)  .$
\end{corollary}

\begin{proof}
Clearly $f_{n}\left(  x\right)  >0$ for all $x\in(0,j_{n+\frac{1}{2},1})$ and
$f_{n}^{\prime}\left(  x\right)  $ is negative for $x\in(j_{n-\frac{1}{2}%
,1},j_{n+\frac{1}{2},1}).$ Hence the summands in expression in (\ref{eqs4})
are positive.
\end{proof}

\section{The derivative of $w\left(  f\right)  $}

In Section 2 we established the useful formula $v^{\prime}+pv=p^{\prime
}f^{\prime}f+q^{\prime}f^{2}.$ For the function $w$ we have
\begin{equation}
w=\left(  p^{\prime}f^{\prime}+q^{\prime}f\right)  ^{2}f-A\cdot
v\label{eqidww}%
\end{equation}
where%
\begin{equation}
A=\left(  p^{\prime\prime}-p^{\prime}p+2q^{\prime}\right)  f^{\prime}+\left(
-2p^{\prime}q+q^{\prime\prime}+pq^{\prime}\right)  f.\label{eqdefA2}%
\end{equation}
It is a natural question whether there a nice formula for $w^{\prime}+pw?$ We
have the following result:

\begin{theorem}
\label{ThmMain3}The following formula holds:
\begin{equation}
w^{\prime}+pw=\left(  p^{\prime}f^{\prime}+q^{\prime}f\right)  \cdot\left(
\left(  p^{\prime\prime}+q^{\prime}\right)  f^{\prime}f+q^{\prime\prime}%
f^{2}+p^{\prime}f^{\prime}f^{\prime}\right)  -A^{\prime}\cdot
v.\label{eqidwprime}%
\end{equation}
In the case $q^{\prime}=0$ this simplifies to
\[
w^{\prime}+pw=p^{\prime}f^{\prime2}\left(  p^{\prime\prime}f+p^{\prime
}f^{\prime}\right)  -A^{\prime}v.
\]

\end{theorem}

\begin{proof}
At first we differentiate (\ref{eqidww}) and we obtain
\[
w^{\prime}+pw=\frac{d}{dx}\left[  \left(  p^{\prime}f^{\prime}+q^{\prime
}f\right)  ^{2}f\right]  -A^{\prime}v-Av^{\prime}+\left(  p^{\prime}f^{\prime
}+q^{\prime}f\right)  ^{2}pf-A\cdot pv.
\]
Next we compute the subsum
\[
S:=-Av^{\prime}+\left(  p^{\prime}f^{\prime}+q^{\prime}f\right)  ^{2}pf-A\cdot
pv=-A\left(  p^{\prime}f^{\prime}+q^{\prime}f\right)  f+\left(  p^{\prime
}f^{\prime}+q^{\prime}f\right)  ^{2}pf
\]
where we used that $v^{\prime}+pv=\left(  p^{\prime}f^{\prime}+q^{\prime
}f\right)  \cdot f$, see (\ref{eq1}). Using (\ref{eqdefA2}) we conclude that
\begin{align*}
S  & =-\left(  p^{\prime}f^{\prime}+q^{\prime}f\right)  \cdot f\cdot\left(
A-\left(  p^{\prime}pf^{\prime}+pq^{\prime}f\right)  \right) \\
& =-\left(  p^{\prime}f^{\prime}+q^{\prime}f\right)  \cdot f\cdot\left(
\left(  p^{\prime\prime}-p^{\prime}p+2q^{\prime}\right)  f^{\prime}+\left(
-2p^{\prime}q+q^{\prime\prime}+pq^{\prime}\right)  f-\left(  p^{\prime
}pf^{\prime}+pq^{\prime}f\right)  \right) \\
& =-\left(  p^{\prime}f^{\prime}+q^{\prime}f\right)  \cdot\left[  \left(
p^{\prime\prime}-2p^{\prime}p+2q^{\prime}\right)  f^{\prime}f+\left(
-2p^{\prime}q+q^{\prime\prime}\right)  f^{2}\right]  .
\end{align*}
Further
\begin{align*}
& \frac{d}{dx}\left[  \left(  p^{\prime}f^{\prime}+q^{\prime}f\right)
^{2}f\right] \\
& =2\left(  p^{\prime}f^{\prime}+q^{\prime}f\right)  \left(  p^{\prime\prime
}f^{\prime}+p^{\prime}f^{\prime\prime}+q^{\prime\prime}f+q^{\prime}f^{\prime
}\right)  f+\left(  p^{\prime}f^{\prime}+q^{\prime}f\right)  ^{2}f^{\prime}\\
& =\left(  p^{\prime}f^{\prime}+q^{\prime}f\right)  \cdot\left(
2p^{\prime\prime}f^{\prime}f+2p^{\prime}f^{\prime\prime}f+2q^{\prime\prime
}ff+2q^{\prime}f^{\prime}f+p^{\prime}f^{\prime}f^{\prime}+q^{\prime}f^{\prime
}f\right) \\
& =\left(  p^{\prime}f^{\prime}+q^{\prime}f\right)  \cdot\left(  \left(
2p^{\prime\prime}+3q^{\prime}\right)  f^{\prime}f+2p^{\prime}f^{\prime\prime
}f+2q^{\prime\prime}ff+p^{\prime}f^{\prime}f^{\prime}\right) \\
& =\left(  p^{\prime}f^{\prime}+q^{\prime}f\right)  \cdot\left(  \left(
2p^{\prime\prime}+3q^{\prime}\right)  f^{\prime}f-2p^{\prime}pf^{\prime
}f-2p^{\prime}qf^{2}+2q^{\prime\prime}ff+p^{\prime}f^{\prime}f^{\prime}\right)
\\
& =\left(  p^{\prime}f^{\prime}+q^{\prime}f\right)  \cdot\left[  \left(
2p^{\prime\prime}-2p^{\prime}p+3q^{\prime}\right)  f^{\prime}f+\left(
2q^{\prime\prime}-2p^{\prime}q\right)  ff+p^{\prime}f^{\prime}f^{\prime
}\right]  .
\end{align*}
Note that
\[
w^{\prime}+pw=\frac{d}{dx}\left[  \left(  p^{\prime}f^{\prime}+q^{\prime
}f\right)  ^{2}f\right]  +S-A^{\prime}v.
\]
We can factor out $\left(  p^{\prime}f^{\prime}+q^{\prime}f\right)  $ in the
first two summands. It follows that
\[
\frac{d}{dx}\left[  \left(  p^{\prime}f^{\prime}+q^{\prime}f\right)
^{2}f\right]  +S=\left(  p^{\prime}f^{\prime}+q^{\prime}f\right)  \cdot\left(
\left(  p^{\prime\prime}+q^{\prime}\right)  f^{\prime}f+q^{\prime\prime}%
f^{2}+p^{\prime}f^{\prime}f^{\prime}\right)  .
\]
using that
\begin{align*}
\left(  2p^{\prime\prime}-2p^{\prime}p+3q^{\prime}\right)  -\left(
p^{\prime\prime}-2p^{\prime}p+2q^{\prime}\right)   & =p^{\prime\prime
}+q^{\prime}\\
\left(  2q^{\prime\prime}-2p^{\prime}q\right)  -\left(  -2p^{\prime
}q+q^{\prime\prime}\right)   & =q^{\prime\prime}.
\end{align*}
The proof is complete.
\end{proof}

The following is a key result in the paper:\ 

\begin{theorem}
\label{ThmMain4}Assume that $q^{\prime}=0.$ Then the following identity
\[
w^{\prime}+\frac{3}{2}\left(  p-\frac{p^{\prime\prime}}{p^{\prime}}\right)
w=p^{\prime}\cdot f^{\prime}\cdot V
\]
holds where
\[
V:=\left(  \frac{1}{2}\left(  p^{\prime}p-p^{\prime\prime}\right)
f+p^{\prime}f^{\prime}\right)  f^{\prime}-\left(  \frac{1}{2}p^{2}-p^{\prime
}-2q+\frac{p^{\prime\prime\prime}}{p^{\prime}}-\frac{3}{2}\frac{(p^{\prime
\prime})^{2}}{p^{\prime2}}\right)  v.
\]

\end{theorem}

\begin{proof}
We use the formula (\ref{eqidwprime}) for $w^{\prime}+pw$ and (\ref{eqidww})
for the case $q^{\prime}=0$ to compute
\begin{align*}
& w^{\prime}+pw+\left(  \frac{1}{2}p-\frac{3}{2}\frac{p^{\prime\prime}%
}{p^{\prime}}\right)  w\\
& =p^{\prime}f^{\prime2}\left(  p^{\prime\prime}f+p^{\prime}f^{\prime}\right)
-A^{\prime}v+\left(  \frac{1}{2}p-\frac{3}{2}\frac{p^{\prime\prime}}%
{p^{\prime}}\right)  \left(  p^{\prime}f^{\prime}\right)  ^{2}f-\left(
\frac{1}{2}p-\frac{3}{2}\frac{p^{\prime\prime}}{p^{\prime}}\right)  Av.
\end{align*}
Note that we simplify the sum of the first and third summand by factoring our
$p^{\prime}f^{\prime2}$ and computing
\[
\left(  p^{\prime\prime}f+p^{\prime}f^{\prime}\right)  +\left(  \frac{1}%
{2}p-\frac{3}{2}\frac{p^{\prime\prime}}{p^{\prime}}\right)  p^{\prime}%
f=\frac{1}{2}\left(  p^{\prime}p-p^{\prime\prime}\right)  f+p^{\prime
}f^{\prime}\
\]
leading to the first summand in $V.$ Since $A=\left(  p^{\prime\prime
}-p^{\prime}p\right)  f^{\prime}-2p^{\prime}qf$ and $f^{\prime\prime
}=-pf^{\prime}-qf$ we obtain that%
\[
A^{\prime}=\left(  p^{\prime\prime\prime}-p^{\prime}p^{\prime}-p^{\prime
\prime}p\right)  f^{\prime}+\left(  p^{\prime\prime}-p^{\prime}p\right)
\left(  -pf^{\prime}-qf\right)  -2p^{\prime\prime}qf-2p^{\prime}qf^{\prime}.
\]
We conclude that
\begin{align*}
& A^{\prime}+\left(  \frac{1}{2}p-\frac{3}{2}\frac{p^{\prime\prime}}%
{p^{\prime}}\right)  A\\
& =\left(  p^{\prime\prime\prime}-p^{\prime}p^{\prime}-p^{\prime\prime
}p\right)  f^{\prime}+\left(  p^{\prime\prime}-p^{\prime}p\right)  \left(
-pf^{\prime}\right)  -2p^{\prime}qf^{\prime}+\left(  \frac{1}{2}p-\frac{3}%
{2}\frac{p^{\prime\prime}}{p^{\prime}}\right)  \left(  p^{\prime\prime
}-p^{\prime}p\right)  f^{\prime}\\
& +\left(  p^{\prime\prime}-p^{\prime}p\right)  \left(  -qf\right)
-2p^{\prime\prime}qf-\left(  \frac{1}{2}p-\frac{3}{2}\frac{p^{\prime\prime}%
}{p^{\prime}}\right)  2p^{\prime}qf.
\end{align*}
The crucial point is the fact that
\[
\left(  p^{\prime\prime}-p^{\prime}p\right)  \left(  -q\right)  -2p^{\prime
\prime}q-\left(  \frac{1}{2}p-\frac{3}{2}\frac{p^{\prime\prime}}{p^{\prime}%
}\right)  2p^{\prime}q=0.
\]
Further we can simplify the coefficient for $f^{\prime}$ in $A^{\prime
}+\left(  \frac{1}{2}p-\frac{3}{2}\frac{p^{\prime\prime}}{p^{\prime}}\right)
$ by observing that
\begin{align*}
& \left(  p^{\prime\prime\prime}-p^{\prime}p^{\prime}-p^{\prime\prime
}p\right)  +\left(  p^{\prime\prime}-p^{\prime}p\right)  \left(  -p\right)
-2p^{\prime}q+\left(  \frac{1}{2}p-\frac{3}{2}\frac{p^{\prime\prime}%
}{p^{\prime}}\right)  \left(  p^{\prime\prime}-p^{\prime}p\right) \\
& =p^{\prime\prime\prime}-p^{\prime}p^{\prime}-p^{\prime\prime}p-\frac{1}%
{2}p^{\prime\prime}p+p^{\prime}p^{2}-2p^{\prime}q-\frac{3}{2}\frac
{p^{\prime\prime2}}{p^{\prime}}+\frac{3}{2}p^{\prime\prime}p\\
& =p^{\prime\prime\prime}-p^{\prime}p^{\prime}+\frac{1}{2}p^{\prime}%
p^{2}-2p^{\prime}q-\frac{3}{2}\frac{p^{\prime\prime2}}{p^{\prime}}.
\end{align*}
Thus we have
\[
A^{\prime}+\left(  \frac{1}{2}p-\frac{3}{2}\frac{p^{\prime\prime}}{p^{\prime}%
}\right)  A=f^{\prime}\left(  \frac{1}{2}p^{2}p^{\prime}-(p^{\prime}%
)^{2}-2p^{\prime}q+p^{\prime\prime\prime}-\frac{3}{2}\frac{(p^{\prime\prime
})^{2}}{p^{\prime}}\right)  .
\]
The proof is completed by dividing the last expression by $p^{\prime}.$
\end{proof}

\section{Positivity of the function $V\left(  f\right)  $}

Assume that $q^{\prime}=0.$ Then Theorem \ref{ThmMain4} shows that
\begin{equation}
w^{\prime}+\frac{3}{2}\left(  p-\frac{p^{\prime\prime}}{p^{\prime}}\right)
w=p^{\prime}\cdot f^{\prime}\cdot V\label{eqwagain}%
\end{equation}
where
\begin{align}
V  & =p^{\prime}f^{\prime2}+\frac{1}{2}\left(  p^{\prime}p-p^{\prime\prime
}\right)  f^{\prime}f-Bv\label{eqVagain}\\
B  & =\frac{1}{2}p^{2}-p^{\prime}-2q+\frac{p^{\prime\prime\prime}}{p^{\prime}%
}-\frac{3}{2}\frac{(p^{\prime\prime})^{2}}{\left(  p^{\prime}\right)  ^{2}%
}.\label{eqBagain}%
\end{align}
In the following we want to discuss the positivity of $V.$ More general, we
consider functions $F$ of the form
\[
F=a_{1}f^{\prime2}+a_{2}f^{\prime}f+a_{3}v.
\]
Setting $a_{1}=p^{\prime}$ and $a_{2}=\frac{1}{2}\left(  p^{\prime}%
p-p^{\prime\prime}\right)  $ and $a_{3}=-B$ we obtain the function $V.$

\begin{proposition}
\label{ThmMain6}Assume that $q^{\prime}=0$ and let $f$ be a solution of
$f^{\prime\prime}+pf^{\prime}+qf=0$, and define $P=\int p\left(  x\right)
dx.$ Let $a_{1},a_{2},a_{3}\in C^{1}\left(  a,b\right)  $ and define
$F=a_{1}f^{\prime2}+a_{2}f^{\prime}f+a_{3}v.$ Then%
\[
e^{-P\left(  x\right)  }F^{\prime}\left(  x\right)  =A_{1}f^{\prime2}%
+A_{2}f^{\prime}f+A_{3}v
\]
where is equal to
\begin{align*}
A_{1}  & =a_{1}^{\prime}-a_{1}p+2a_{2}\\
A_{2}  & =a_{2}^{\prime}-2a_{1}q+pa_{2}+a_{3}p^{\prime}\\
A_{3}  & =a_{3}^{\prime}-a_{2}.
\end{align*}

\end{proposition}

\begin{proof}
We have
\[
F^{\prime}=a_{1}^{\prime}f^{\prime2}+a_{2}^{\prime}f^{\prime}f+a_{3}^{\prime
}v+2a_{1}f^{\prime}f^{\prime\prime}+a_{2}\left(  f^{\prime}f^{\prime
}+f^{\prime\prime}f\right)  +a_{3}v^{\prime}.
\]
Note that $f^{\prime}f^{\prime\prime}=-pf^{\prime}f^{\prime}-qf^{\prime}f$ and
$f^{\prime}f^{\prime}+f^{\prime\prime}f=2f^{\prime}f^{\prime}-v $. Hence
\[
F^{\prime}=\left(  a_{1}^{\prime}-2a_{1}p+2a_{2}\right)  f^{\prime2}+\left(
a_{2}^{\prime}-2a_{1}q\right)  f^{\prime}f+a_{3}^{\prime}v-a_{2}%
v+a_{3}v^{\prime}.
\]
Now we want to compute $F^{\prime}+pF$ and observe that $a_{3}v^{\prime
}+pa_{3}v=a_{3}\left(  v+pv\right)  =a_{3}p^{\prime}f^{\prime}f.$ It follows
that
\[
F^{\prime}+pF=\left(  a_{1}^{\prime}-a_{1}p+2a_{2}\right)  f^{\prime2}+\left(
a_{2}^{\prime}-2a_{1}q+pa_{2}+a_{3}p^{\prime}\right)  f^{\prime}f+\left(
a_{3}^{\prime}-a_{2}\right)  v.
\]
The proof is now accomplished since
\[
e^{-P\left(  x\right)  }F^{\prime}\left(  x\right)  =e^{-P\left(  x\right)
}\frac{d}{dx}\left(  e^{P\left(  x\right)  }F\left(  x\right)  \right)
=F^{\prime}\left(  x\right)  +p\left(  x\right)  F\left(  x\right)
\]
for all $x\in\left(  a,b\right)  .$
\end{proof}

We specialize now to the case of the function $V$: then $a_{1}=p^{\prime}$ and
$a_{2}=\frac{1}{2}\left(  p^{\prime}p-p^{\prime\prime}\right)  $ and
$a_{3}=-B$. Then
\[
A_{1}=a_{1}^{\prime}-a_{1}p+2a_{2}=p^{\prime\prime}-p^{\prime}p+2\frac{1}%
{2}\left(  p^{\prime}p-p^{\prime\prime}\right)  =0.
\]
Further a straight forward calculation shows that
\begin{align*}
A_{2}  & =a_{2}^{\prime}-2a_{1}q+pa_{2}+a_{3}p^{\prime}=\frac{1}{2}p^{\prime
}p^{\prime}+\frac{1}{2}p^{\prime\prime}p-\frac{1}{2}p^{\prime\prime\prime
}-2p^{\prime}q+\frac{1}{2}\left(  p^{\prime}p^{2}-p^{\prime\prime}p\right) \\
& -\frac{1}{2}p^{\prime}p^{2}+p^{\prime}p^{\prime}+2p^{\prime}q-p^{\prime
\prime\prime}+\frac{3}{2}\frac{(p^{\prime\prime})^{2}}{p^{\prime}},
\end{align*}
is given by
\begin{equation}
A_{2}=\frac{3}{2}\left(  (p^{\prime})^{2}-p^{\prime\prime\prime}%
+\frac{(p^{\prime\prime})^{2}}{p^{\prime}}\right)  .\label{eqnewA2}%
\end{equation}
Finally (using that $q^{\prime}=0)$
\begin{align*}
A_{3}  & =a_{3}^{\prime}-a_{2}=\frac{d}{dx}\left(  -\frac{1}{2}p^{2}%
+p^{\prime}+2q-\frac{p^{\prime\prime\prime}}{p^{\prime}}+\frac{3}{2}%
\frac{(p^{\prime\prime})^{2}}{\left(  p^{\prime}\right)  ^{2}}\right)
-\frac{1}{2}\left(  p^{\prime}p-p^{\prime\prime}\right) \\
& =-p^{\prime}p+p^{\prime\prime}-\frac{p^{\left(  4\right)  }p^{\prime
}-p^{\prime\prime\prime}p^{\prime\prime}}{p^{\prime2}}+3\frac{p^{\prime\prime
}p^{\prime\prime\prime}\left(  p^{\prime}\right)  -\left(  p^{\prime\prime
}\right)  ^{2}p^{\prime\prime}}{\left(  p^{\prime}\right)  ^{3}}-\frac{1}%
{2}\left(  p^{\prime}p-p^{\prime\prime}\right)  ,
\end{align*}
is given by
\begin{equation}
A_{3}=\frac{3}{2}\left(  p^{\prime\prime}-p^{\prime}p\right)  -\frac
{p^{\left(  4\right)  }}{p^{\prime}}+\frac{4p^{\prime\prime}p^{\prime
\prime\prime}}{p^{\prime2}}-\frac{3\left(  p^{\prime\prime}\right)  ^{3}%
}{\left(  p^{\prime}\right)  ^{3}}.\label{eqnewA3}%
\end{equation}
It is remarkable that the coefficients $A_{2}$ and $A_{3}$ given by
(\ref{eqnewA2}) and (\ref{eqnewA3}) do not depend on the value (or sign) of
the constant $q$. In summary, we have the identity
\begin{equation}
V^{\prime}+pV=A_{2}f^{\prime}f+A_{3}v\left(  f\right)  .\label{eqnewAA}%
\end{equation}

The following is an improvement of Theorem \ref{ThmMain4}.

\begin{corollary}
Assume $f\in C^{2}\left[  a,b\right]  $ satisfies $f^{\prime\prime}%
+pf^{\prime}+qf=0$ and $f\left(  a\right)  =f^{\prime}\left(  a\right)  =0.$
Assume that $q^{\prime}=0$ and define $P=\int p\left(  x\right)  dx$. Then the
following identity
\[
w^{\prime}+\frac{3}{2}\left(  p-\frac{p^{\prime\prime}}{p^{\prime}}\right)
w=p^{\prime}\cdot f^{\prime}\cdot V
\]
holds where $V\left(  x\right)  $ is given for $x\in\left[  a,b\right]  $ by%
\begin{align*}
V\left(  x\right)   & =\frac{1}{2}f\left(  x\right)  ^{2}A_{2}\left(
x\right)  -e^{-P\left(  x\right)  }\int_{a}^{x}\frac{1}{2}f\left(  t\right)
^{2}\left(  A_{2}p+A_{2}^{\prime}\right)  e^{P}dt\\
& +e^{-P\left(  x\right)  }\int_{a}^{x}A_{3}\left(  t\right)  v\left(
f\right)  \left(  t\right)  e^{P\left(  t\right)  })dt
\end{align*}
and $A_{2}$ and $A_{3}$ are given by (\ref{eqnewA2}) and (\ref{eqnewA3}).

In particular, if $A_{2}\geq0$ and $A_{3}\geq0$ and $A_{2}p+A_{2}^{\prime}%
\leq0$ on $\left[  a,b\right]  $ then $V\geq0$ on $\left[  a,b\right]  .$
\end{corollary}

\begin{proof}
From (\ref{eqnewAA}) we see that
\[
\frac{d}{dx}\left(  e^{P}V\right)  =e^{P}\left(  V^{\prime}+pV\right)
=A_{2}f^{\prime}fe^{P}+A_{3}v\left(  f\right)  e^{P}.
\]
Integration shows that
\[
V\left(  x\right)  -V\left(  0\right)  =e^{-P\left(  x\right)  }\int_{a}%
^{x}(A_{2}\left(  t\right)  f^{\prime}\left(  t\right)  f\left(  t\right)
e^{P\left(  t\right)  }+A_{3}\left(  t\right)  v\left(  f\right)  \left(
t\right)  e^{P\left(  t\right)  })dt.
\]
Partial integration shows that
\[
\int_{a}^{x}A_{2}\left(  t\right)  f^{\prime}\left(  t\right)  f\left(
t\right)  e^{P\left(  t\right)  }dt=\frac{1}{2}f^{2}\left(  x\right)
A_{2}\left(  x\right)  e^{P\left(  x\right)  }-\int_{a}^{x}\frac{1}{2}%
f^{2}\left(  A_{2}p+A_{2}^{\prime}\right)  e^{P}dt.
\]

\end{proof}

Now we specialize to the case $f=f_{n}$ and since $p\left(  x\right)  =\left(
-2n\right)  /x$ we have
\begin{align*}
A_{2}  & =\frac{3}{2}\left(  (p^{\prime})^{2}-p^{\prime\prime\prime}%
+\frac{(p^{\prime\prime})^{2}}{p^{\prime}}\right)  =\frac{6n\left(
n-1\right)  }{x^{4}}\\
A_{3}  & =\frac{3}{2}\left(  p^{\prime\prime}-p^{\prime}p\right)
-\frac{p^{\left(  4\right)  }}{p^{\prime}}+\frac{4p^{\prime\prime}%
p^{\prime\prime\prime}}{p^{\prime2}}-\frac{3\left(  p^{\prime\prime}\right)
^{3}}{\left(  p^{\prime}\right)  ^{3}}=\frac{6n\left(  n-1\right)  }{x^{3}}.
\end{align*}
Further
\[
\left(  A_{2}p+A_{2}^{\prime}\right)  =\frac{6n\left(  n-1\right)  }{x^{4}%
}\frac{-2n}{x}+\frac{6n\left(  n-1\right)  }{x^{5}}\left(  -4\right)
=-\frac{6n\left(  n-1\right)  }{x^{5}}\left(  2n+4\right)  .
\]
Further $P\left(  x\right)  =-2n\ln x$ and $e^{-P\left(  x\right)  }=x^{2n}$,
so we have
\begin{equation}
\frac{V\left(  f_{n}\right)  \left(  x\right)  }{6n\left(  n-1\right)  }%
=\frac{f_{n}^{2}\left(  x\right)  }{2x^{4}}+\left(  n+2\right)  x^{2n}\int
_{0}^{x}\frac{f_{n}\left(  t\right)  ^{2}}{t^{2n+5}}dt+x^{2n}\int_{0}^{x}%
\frac{v\left(  f_{n}\right)  \left(  t\right)  }{t^{2n+3}}%
dt.\label{eqVpositive}%
\end{equation}

\begin{corollary}
\label{Cor6a}Let $f_{n}\left(  x\right)  $ be defined as in (\ref{eqdeffn}).
Then $V\left(  f_{n}\right)  $ is positive on $\left(  0,\infty\right)  $, and
the functions $x^{-2n}V\left(  f_{n}\right)  $ and $V\left(  f_{n}\right)  $
are increasing and positive for all $x\in\left(  0,j_{n+\frac{1}{2},1}%
^{\prime}\right)  $.
\end{corollary}

\begin{proof}
Clearly (\ref{eqVpositive}) implies that $V\left(  f_{n}\right)  \left(
x\right)  \geq0$ for all $x>0.$ Using that $q=1$ and $p=-2n/x,$ (\ref{eqnewAA}%
) shows that
\begin{equation}
V\left(  f_{n}\right)  ^{\prime}-\frac{2n}{x}V\left(  f_{n}\right)
=\frac{6n\left(  n-1\right)  }{x^{4}}f_{n}^{\prime}f_{n}+\frac{6n\left(
n-1\right)  }{x^{3}}v\left(  f_{n}\right)  .\label{eqVmain}%
\end{equation}
Since for any differentiable function $g$
\begin{equation}
x^{-\beta}\frac{d}{dx}\left(  x^{\beta}g\left(  x\right)  \right)  =g^{\prime
}\left(  x\right)  +\frac{\beta}{x}g\left(  x\right) \label{eqbeta}%
\end{equation}
we see that $F\left(  x\right)  :=x^{-2n}V\left(  f_{n}\right)  $ is
increasing on $\left(  0,j_{n+\frac{1}{2},1}^{\prime}\right)  $. Since
$F\left(  0\right)  \geq0$ we conclude that $F$ is positive on $\left(
0,j_{n+\frac{1}{2},1}^{\prime}\right)  .$ Since $V\left(  f_{n}\right)  $ is
the product of two increasing and positive functions on $\left(
0,j_{n+\frac{1}{2},1}^{\prime}\right)  $ it follows that $V\left(
f_{n}\right)  $ is increasing and positive on $\left(  0,j_{n+\frac{1}{2}%
,1}^{\prime}\right)  .$
\end{proof}

Finally we can prove the main result of the paper:

\begin{corollary}
Let $f_{n}\left(  x\right)  $ be defined as in (\ref{eqdeffn}). Then
$x^{-3\left(  n-1\right)  }w\left(  f_{n}\right)  $ and $w\left(
f_{n}\right)  $ are increasing and positive for all $x\in\left(
0,j_{1}\left(  f_{n}^{\prime}\right)  \right)  .$ Moreover $w\left(
f_{n}\right)  $ is positive on $\left(  0,j_{1}\left(  f_{n}\right)  \right)
.$
\end{corollary}

\begin{proof}
We apply Theorem \ref{ThmMain4} with $p=\left(  -2n\right)  /x$ and $q=1.$
Then $w=w\left(  f_{n}\right)  $ satisfies the equation
\[
w^{\prime}+\frac{3}{2}\left(  p-\frac{p^{\prime\prime}}{p^{\prime}}\right)
w=w^{\prime}-\frac{3\left(  n-1\right)  }{x}w=p^{\prime}\cdot f_{n}^{\prime
}\cdot V\left(  f_{n}\right)  .
\]
Since $p^{\prime}\left(  x\right)  =2n/x^{2}\geq0$ and $V\left(  f_{n}\right)
\left(  x\right)  \geq0$ for all $x>0,$ and $f_{n}^{\prime}\left(  x\right)
>0 $ for all $\left(  0,j_{1}\left(  f_{n}^{\prime}\right)  \right)  $ we
conclude the left hand side is positive. Using (\ref{eqbeta}) we conclude that
$x^{-3\left(  n-1\right)  }w$ is increasing on $\left(  0,j_{1}\left(
f_{n}^{\prime}\right)  \right)  $, and it is positive since $w\left(
0\right)  =0.$ For the positivity of $w\left(  f_{n}\right)  $ on the interval
$\left(  j_{1}\left(  f_{n}^{\prime}\right)  ,j_{1}\left(  f_{n}\right)
\right)  $ see Corollary \ref{Cor4}.
\end{proof}

$\ $

\textbf{Acknowlegdement} The work of all authors was funded under project
KP-06-N52-1 with Bulgarian NSF.

\smallskip\noindent\textbf{Data availability} Data sharing is not applicable
to this article since no data sets were generated or analyzed.

\smallskip\noindent\textbf{Declarations}

\smallskip\noindent\textbf{Conflict of interest} The authors declare that they
have no competing interests.


\begin{thebibliography}{99}                                                                                               %
\bibitem {AMR19}Ait-Haddou, R., Mazure, M.-L., Ruhland, H.: A remarkable
Wronskian with application to critical lengths of cycloidal spaces. Calcolo
\textbf{56} (2019) 56:45

\bibitem {AKR07}Aldaz, J. M., Kounchev, O., Render, H.: Bernstein operators
for exponential polynomials. Constr. Approx. \textbf{29}, 345--367 (2009)

\bibitem {AKR08b}Aldaz, J. M., Kounchev, O., Render, H.: Shape preserving
properties of Bernstein operators on Extended Chebyshev spaces, Numer. Math.
\textbf{114, } 1--25 (2009)

\bibitem {AKR08}Aldaz, J. M., Kounchev, O., Render, H.: Bernstein operators
for extended Chebyshev systems. Appl. Math. Comput. \textbf{217}, 790--800 (2010)

\bibitem {BCM20}Beccari, C.V, Casciola, G., Mazure, M.L.: Critical length: An
alternative approach. J. Comput. Appl. Math. \textbf{370, } (2020)

\bibitem {CMP94}Carnicer, J.M., Pe\~{n}a, J.M.: Totally positive bases for
shape preserving curve design and optimality of B-splines. Comput. Aided Geom.
Des. \textbf{11}, 635--656 (1994)

\bibitem {CMP04}Carnicer, J.M., Mainar, E., Pe\~{n}a, J.M.: Critical Length
for Design Purposes and Extended Chebyshev Spaces. Constr. Approx.
\textbf{20}, 55--71 (2004)

\bibitem {CMP14}Carnicer, J.M., Mainar, E., Pe\~{n}a, J.M.: On the critical
length of cycloidal spaces. Constr. Approx. \textbf{39}, 573--583 (2014)

\bibitem {CMP17}Carnicer, J.M., Mainar, E., Pe\~{n}a, J.M.: Critical lengths
of cycloidal spaces are zeros of Bessel functions, Calcolo \textbf{54},
1521--1531 (2017)

\bibitem {CMP23}Carnicer, J.M., Mainar, E., Pe\~{n}a, J.M.: Spherical Bessel
Functions and Critical lengths, Rev. Real Acad. Cienc. Exactas Fis. Nat. Ser.
A-Mat. \textbf{117}, (2023)

\bibitem {CLM}Costantini, P., Lyche, T., Manni, C.: On a class of weak
Tchebycheff systems, Numer. Math. \textbf{101}, 333--354 (2005)

\bibitem {oggybook}Kounchev, O.I.: Multivariate Polysplines. Applications to
Numerical and Wavelet Analysis. Academic Press, London--San Diego, 2001.

\bibitem {Pena97}Pe\~{n}a, J.M.: Shape preserving representations for
trigonometric polynomial curves. Comput. Aided Geom. Des. \textbf{14}, 5--11 (1997)

\bibitem {Watson}Watson, G.N.: A Treatise on the Theory of Bessel Functions.
Cambridge University Press, Cambridge(1922)

\bibitem {Zhan96}Zhang, J.: C-curves: an extension of cubic curves. Comput.
Aided Geom. Design 13 (1996), 199--217.
\end{thebibliography}
\end{document}